\documentclass{amsart}
\usepackage{amssymb, amsmath, latexsym, mathabx, dsfont}

\renewcommand{\baselinestretch}{\baselinestretch}
\renewcommand{\baselinestretch}{1.1}
\numberwithin{equation}{section}

\newtheorem{thm}{Theorem}[section]
\newtheorem{lem}[thm]{Lemma}

\theoremstyle{definition}

\theoremstyle{remark}

\newcommand{\ra}{{\,\rightarrow\,}}

\newcommand{\nra}{{\,\nrightarrow\,}}

\newcommand{\z}{{\mathbb Z}}

\begin{document}

\title{Universal sums of generalized octagonal numbers}
\author{Jangwon Ju and Byeong-Kweon Oh }

\address{Department of Mathematical Sciences, Seoul National University, Seoul 08826, Korea}
\email{jjw@snu.ac.kr}
\thanks{This work of the first author was supported by BK21 PLUS SNU Mathematical Sciences Division.}

\address{Department of Mathematical Sciences and Research Institute of Mathematics, Seoul National University, Seoul 08826, Korea}
\email{bkoh@snu.ac.kr}
\thanks{This work of the second author was supported by the National Research Foundation of Korea (NRF-2017R1A2B4003758).}



\begin{abstract}  An integer of the form $P_8(x)=3x^2-2x$ for some integer $x$ is called a generalized octagonal number. A quaternary sum $\Phi_{a,b,c,d}(x,y,z,t)=aP_8(x)+bP_8(y)+cP_8(z)+dP_8(t)$ of generalized octagonal numbers  is called {\it universal} if  $\Phi_{a,b,c,d}(x,y,z,t)=n$ has an integer solution $x,y,z,t$ for any positive integer $n$.  In this article, we show that if $a=1$ and $(b,c,d)=(1,3,3), (1,3,6), (2,3,6), (2,3,7)$ or $(2,3,9)$, then $\Phi_{a,b,c,d}(x,y,z,t)$ is universal. These were conjectured  by Sun in  \cite {sun}. We also give an effective criterion on the universality of an arbitrary sum $a_1 P_8(x_1)+a_2P_8(x_2)+\cdots+a_kP_8(x_k)$ of  generalized octagonal numbers , which is a generalization of ``$15$-theorem" of Conway and Schneeberger.   
\end{abstract}

\maketitle

\section{introduction}
The famous Lagrange's four square theorem states that every positive integer can be written as a sum of at most four integral squares. 
Motivated by Lagrange's four square theorem, Ramanujan provided a list of 55 candidates of  diagonal quaternary quadratic forms that represent all positive integers. In \cite{dick}, Dickson pointed out that the quaternary form $x^2+2y^2+5z^2+5t^2$ that is included in Ramanujan's list represents all positive integers except $15$ and confirmed that Ramanujan's assertion for all the other $54$ forms is true. A positive definite integral quadratic form is called {\it universal} if it represents all non-negative integers.  
The problem on determining all universal quaternary forms was completed by Conway and Schneeberger. They proved that there are exactly $204$ universal quaternary quadratic forms. Furthermore, they  proved the so called ``15-theorem", which states that every positive definite integral quadratic form that represents $1,2,3,5,6,7,10,14$ and $15$ is, in fact, universal, irrespective of its rank (see \cite{b}). Recently, Bhargava and Hanke \cite {BH} proved, so called, ``290-theorem", which states that    
every positive definite integer-valued quadratic form is universal if it represents 
$$
\begin{array} {ll}
1,\ 2, \ 3,\ 5,\ 6,\ 7,\ 10,\ 13,\ 14,\ 15,\ 17,\ 19,\ 21,\ 22,\ 23,\ 26,\ 29,\\
30,\ 31,\ 34,\ 35,\ 37,\ 42,\ 58,\ 93,\ 110,\ 145,\ 203, \ \  \text{and} \  \ 290.
\end{array}
$$

  In general,   a {\it polygonal number of order $m$} (or {\it an $m$-gonal number}) for $m\ge3$ is defined by
$$
P_m(x)= \frac{(m-2)x^2-(m-4)x}{2}
$$
for some non-negative integer $x$.  If we admit that $x$ is a  negative integer, then $P_m(x)$ is called  
{\it a generalized polygonal number of order $m$} (or {\it a generalized $m$-gonal number}). 
Then, Lagrange's four square theorem implies that 
the diophantine equation 
$$
P_4(x)+P_4(y)+P_4(z)+P_4(t)=n
$$  
has an integer solution $x,y,z$ and $t$, for any non-negative integer $n$. 

Recently, Sun in \cite{sun} proved that every positive integer can be written as a sum of four generalized octagonal numbers, which is also considered as a generalization of Lagrange's four square theorem. He also defined that for positive integers $a\leq b\leq c\leq d$, a quaternary sum  $aP_8(x)+bP_8(y)+cP_8(z)+dP_8(t)$ (simply, $\Phi_{a,b,c,d}(x,y,z,t)$) of generalized octagonal numbers is {\it universal over $\mathbb Z$}  if the diophantine equation
$$
\Phi_{a,b,c,d}(x,y,z,t)=aP_8(x)+bP_8(y)+cP_8(z)+dP_8(t)=n
$$
has an integer solution for any non-negative integer $n$. Then, he showed that if $\Phi_{a,b,c,d}$ is universal over $\mathbb Z$, then $a=1$, and $(b,c,d)$ is one of the following 40 triples:
$$
{\setlength\arraycolsep{2pt}
\begin{array}{lllllllll}
&(1,1,1),&(1,1,2),&(1,1,3),&(1,1,4),&(1,2,2),&(1,2,3),&(1,2,4),&(1,2,5),\\
&(1,2,6),&(1,2,7),&(1,2,8),&(1,2,9),&(1,2,10),&(1,2,11),&(1,2,12),&(1,2,13),\\
&(1,3,3),&(1,3,5),&(1,3,6),&(2,2,2),&(2,2,3),&(2,2,4),&(2,2,5),&(2,2,6),\\
&(2,3,4),&(2,3,5),&(2,3,6),&(2,3,7),&(2,3,8),&(2,3,9),&(2,4,4),&(2,4,5),\\
&(2,4,6),&(2,4,7),&(2,4,8),&(2,4,9),&(2,4,10),&(2,4,11),&(2,4,12),&(2,4,13)
\end{array}}
$$
and proved the universalities over $\mathbb Z$ except the following 5 triples:
\begin{equation}\label{candi}
(1,3,3),\ (1,3,6), \ (2,3,6), \ (2,3,7), \ \text{and} \ (2,3,9).
\end{equation}
He conjectured that for each of these 5 triples, the sum $\Phi_{1,b,c,d}$ is also universal over $\mathbb Z$. 

The aim of this article is to prove this conjecture. In Section 3, we further prove that the sum $a_1 P_8(x_1)+a_2P_8(x_2)+\cdots+a_kP_8(x_k)$ of  generalized octagonal numbers is universal over $\mathbb Z$ if and only if it represents 
$$
1,\ 2, \ 3,\  4,\ 6,\  7,\ 9,\ 12,\ 13,\ 14,\ 18, \ \text{and} \  60.
$$ 
 This might be considered as a generalization of ``$15$-theorem" of Conway and Schneeberger. For universal sums of triangular numbers, see \cite{bk}.       

For a quadratic form $f(x_1,x_2,\dots,x_n)=\sum_{1 \le i, j\le n} a_{ij} x_ix_j \ (a_{ij}=a_{ji})$, the corresponding symmetric matrix of $f$ is defined by $M_f=(a_{ij})$. In particular,   
 for a diagonal quadratic form $f(x_1,x_2,\dots,x_n)=a_1x_1^2+a_2x_2^2+\cdots+a_nx_n^2$, we simply write $M_f=\langle a_1,a_2,\dots,a_n\rangle$. 
For an integer $k$, if the diophantine equation $f(x_1,x_2,\dots,x_n)=k$ has an integer solution, then we say $k$ is represented by $f$, and we write $k\ra f$. 
The genus of $f$, denoted by $\text{gen}(f)$, is the set of all quadratic forms that are locally isometric to $f$. 
The number of isometry classes in $\text{gen}(f)$ is called the class number of $f$.

Any unexplained notations and terminologies can be found in \cite{ki} or \cite{om}.

\section{Quaternary universal sums of generalized octagonal numbers}

In this section, we prove that if $(b,c,d)$ is one of the triples in \eqref{candi}, then the quaternary sum $\Phi_{1,b,c,d}$ of generalized octagonal numbers is universal over $\mathbb Z$. 

Let $0<b\leq c\leq d$ be positive integers. Recall that a quaternary sum $\Phi_{1,b,c,d}$ of generalized octagonal numbers is said to be universal over $\mathbb Z$ if the diophantine equation
$$
P_8(x)+bP_8(y)+cP_8(z)+dP_8(t)=n
$$
has an integer solution for any non-negative integer $n$. Note that $\Phi_{1,b,c,d}$ is universal over $\mathbb Z$ if and only if  the equation
$$
(3x-1)^2+b(3y-1)^2+c(3z-1)^2+d(3t-1)^2=3n+1+b+c+d
$$
has an integer solution for any non-negative integer $n$. This is equivalent to the existence of an integer solution $x,y,z,t$ of the diophantine equation
$$
x^2+by^2+cz^2+dt^2=3n+1+b+c+d
$$
such that $xyzt\not\equiv0 \ (\text{mod 3})$.

\begin{thm}
The quaternary sum  $\Phi_{1,1,3,3}$ of generalized octagonal numbers is universal over $\mathbb Z$.
\end{thm}

\begin{proof}

 It is enough to show that $x^2+y^2+3z^2+3t^2=3n+8$ has an integer solution $(x,y,z,t) \in \z^4$ such that $zt \not\equiv  0 \ (\text{mod 3})$.  First, assume that $n=3m$. Note that the class number of $\langle 1,1,6\rangle$ is $1$ and $z^2+t^2+6y^2=3m+2$ has an integer solution such that $zt \not \equiv 0 \ (\text{mod 3})$. 
Hence 
$$
3(3m)+8=(3y+1)^2+(3y-1)^2+3z^2+3t^2, \quad \text{where}  \ \  zt \not\equiv 0 \ (\text{mod 3}).
$$
Furthermore, we also have 
$$
3(3m+2)+8=(3y+2)^2+(3y-2)^2+3z^2+3t^2, \quad \text{where}  \ \  zt \not\equiv 0 \ (\text{mod 3}).
$$
Now assume that $n=3m+1$. Note that for any $m\leq4$, the diophatine equation $x^2+y^2+3z^2+3t^2=3(3m+1)+8$ has an integer solution $(x,y,z,t)=(a,b,c,d)$ such that $zt\not\equiv 0 \ (\text{mod 3})$. Therefore we assume that $m\ge5$.
Similarly as above,  $z^2+t^2+6y^2=3m-13$ has an integer solution $(z,t,y)=(a,b,c) \in \z^3$ such that $ab \not \equiv 0 \ (\text{mod 3})$. Hence 
$$
3(3m+1)+8=(3c-5)^2+(3c+5)^2+3a^2+3b^2, \quad \text{where}  \ \  ab \not\equiv 0 \ (\text{mod 3}).
$$
Therefore, the diophantine equation $(3x-1)^2+(3y-1)^2+3(3z-1)^2+3(3t-1)^2=3n+8$ has always an integer solution.
\end{proof} 

\begin{thm}
The quaternary sum  $\Phi_{1,1,3,6}$ of generalized octagonal numbers is universal over $\mathbb Z$.
\end{thm}
\begin{proof}
 First, we show that $x^2+y^2+9z^2+18t^2=3n+11$ has an integer solution such that $9z^2+18t^2\ne 0$. If $n\leq  20$, then one may directly check that such an integer solution exists. Note that the class number of $\langle 1,1,9\rangle$ is $1$ and it represents all integers $n$ such that 
$n \equiv 2 \ (\text{mod 3})$ and $n \ne 4^s(8t+7)$ for any non-negative integers $s,t$. 
Note that 
for any $n\ge21$,
$$
\begin{aligned} 
&3n+11-18\cdot2^2 \ra \langle 1,1,9\rangle  \quad \text{ if $3n+11 \equiv 1 \ (\text{mod 2})$ and $3n+11 \not \equiv 7 \ (\text{mod 8})$}, \\
&3n+11-18 \ra \langle 1,1,9\rangle  \  \ \quad \quad \text{ if $3n+11 \equiv 7 \ (\text{mod 8})$}, \\
&3n+11-18\cdot2^2 \ra \langle 1,1,9\rangle  \quad \text{ if $3n+11 \equiv 2 \ (\text{mod 4})$}, \\
&3n+11-18 \ra \langle 1,1,9\rangle \  \ \quad  \quad \text{ if $3n+11 \equiv 0 \ (\text{mod 4})$}. \\
\end{aligned}
$$
Hence $3n+11=x^2+y^2+9z^2+18t^2$ has an integer solution $(x,y,z,t)=(a,b,c,d) \in \z^4$ 
such that $ab \not\equiv 0 \ (\text{mod 3})$ and $d=1$ or $2$. Since
$$
a^2+b^2+3(c-2d)^2+6(c+d)^2=a^2+b^2+9c^2+18d^2=3n+11,
$$
$x^2+y^2+3z^2+6t^2=3n+11$ has an integer solution  $(x,y,z,t)=(a,b,c-2d,c+d)$. Now, by  Theorem 9 of \cite{Jones} (see also \cite{Jagy}, and for more generalization see \cite{oh}), there are integers
$e,f$ such that $e^2+2f^2=(c-2d)^2+2(c+d)^2$ and $ef \not \equiv 0 \ (\text{mod 3})$.  
\end{proof}

\begin{thm}
The quaternary sum  $\Phi_{1,2,3,6}$ of generalized octagonal numbers is universal over $\mathbb Z$.
\end{thm}
\begin{proof}
 First, we show that the diophantine equation $x^2+2y^2+9z^2+18t^2=3n+12$ has an integer solution $(x,y,z,t)=(a,b,c,d) \in \z^4$ such that $a^2+2b^2 \ne 0$ and $c^2+2d^2 \ne 0$. If $n\le 44$, then one may directly check that such an integer solution exists. From now on, we assume that $n \ge 45$. 
 
 Assume that an integer $3m$, for some positive integer $m$ is represented by the genus of $\langle 1,2,9\rangle$. Since the class number of $\langle 1,1,2\rangle$ is one, $3m$ is represented by $\langle 1,1,2\rangle$, that is, $x^2+y^2+2z^2=3m$ has an integer solution $(x,y,z)=(a,b,c) \in \z^3$. Note that $a$ or $b$ is divisible by $3$. Therefore $3m$ is represented by $\langle 1,2,9\rangle$. Similarly, every integer of the form $3m$, for some positive integer $m$ that is represented by the genus of $\langle1,2,18\rangle$ is represented by itself.          
 
 If $3n+12 \equiv 1 \ (\text{mod 2})$ and $n\ge 21$, then both $3n+12-18\cdot1^2$ and $3n+12-18\cdot 2^2$ are represented by $\langle 1,2,9\rangle$ by the above observation. Furthermore, at least one of them is not of the form $9\cdot N^2$ for any integer $N$.   If $3n+12 \equiv 2 \ (\text{mod 4})$ and $n\ge 45$, then both $3n+12-9\cdot2^2$ and $3n+12-9\cdot 4^2$ are represented by $\langle 1,2,18\rangle$. Furthermore, at least one of them is not of the form $18\cdot M^2$ for any integer $M$.   
Finally, assume that  $3n+12=2^{\delta}4^ma$, where $a$ is odd, $\delta=0$ or $1$, and $m\ge1$. Then the assertion follows from the fact that the assertion holds if $2^{\delta}a\ge 12$ or $2^{\delta}4^ma=3\cdot4,6\cdot4$ and $9\cdot4$. 

Now, since
$$
a^2+2b^2+3(c-2d)^2+6(c+d)^2=a^2+2b^2+9c^2+18d^2=3n+12,  
$$
and
$$
a^2+2b^2 \ne 0 \quad \text{and} \quad (c-2d)^2+2(c+d)^2 \ne 0,
$$
the diophantine equation $(3x-1)^2+2(3y-1)^2+3(3z-1)^2+6(3t-1)^2=3n+12$ has always an integer solution by Theorem 9 of \cite{Jones}. 
\end{proof}

\begin{thm}\label{1237}
The quaternary sums $\Phi_{1,2,3,7}$ and $\Phi_{1,2,3,9}$ of generalized octagonal numbers  are all universal over $\mathbb Z$.
\end{thm}

\begin{proof}
Since the proofs are quite similar to each other, we only provide the proof of the former case.

It suffices to show that the diophantine equation $x^2+2y^2+3z^2+7t^2=3n+13$ has an integer solution $(x,y,z,t)=(a,b,c,d)\in\mathbb{Z}^4$ such that $abcd\not\equiv0 \ (\text{mod 3})$ for any non-negative integer $n$. If $n\leq 110$, then one may directly check that such an integer solution exists. Therefore we assume that $n\ge111$.
First, we show that the diophantine equation $3x^2+3y^2+6z^2+7t^2=3n+13$ has an integer solution $(x,y,z,t)=(a,b,c,d)\in\mathbb{Z}^4$ such that 
$$
d\not\equiv0 \ (\text{mod 3}),  \  d\equiv 3n+13 \ (\text{mod 2}) \ \text{and}  \  3n+13-7d^2\neq0.
$$
Note that an integer $n$ is represented by   $\langle3,3,6\rangle$  if and only if $n$ is divisible by $3$ and is not of the form $6\cdot4^s(8t+7)$ for any non-negative integers $s,t$. Then one may easily show that there is an integer $d \in \{1,5,7\}$ such that
$3n+13-7\cdot d^2$ is represented by $\langle3,3,6\rangle$.
Furthermore, one may also show that there is an integer $d \in \{2,4\}$ such that
$3n+13-7\cdot d^2$ is represented by $\langle3,3,6\rangle$. Hence By taking a suitable integer $d \in \{1,5,7,2,4\}$, we assume that there are integers $a,b$ and $c$ such that  
$$
3a^2+3b^2+6c^2=3n+13-7\cdot d^2 \equiv 0 \ (\text{mod 2})\quad \text{and}\quad 3n+13-7d^2>0.
$$
Since  $3n+13-7d^2>0$, we may assume that
$$
3a^2+6c^2=(a+2c)^2+2(a-c)^2\neq0.
$$
Hence there are integers $e,f$ such that $e^2+2f^2=(a+2c)^2+2(a-c)^2$ and $ef\not\equiv0 \ (\text{mod 3})$ by Theorem 9 of \cite{Jones}. Therefore,  if $b\not\equiv0 \ (\text{mod 3})$, then we are done.
Assume that $b\equiv0 \ (\text{mod 3})$. Since $e\equiv b \ (\text{mod 2})$,  we have
$$
e^2+3b^2+2f^2=\left(\frac{e+3b}{2}\right)^2+3\cdot\left(\frac{e-b}{2}\right)^2+2f^2=3n+13-7d^2,
$$
where $\left(\frac{e+3b}{2}\right)\left(\frac{e-b}{2}\right) \not \equiv 0 \ (\text{mod 3})$. 
Therefore, the diophantine equation $(3x-1)^2+2(3y-1)^2+3(3z-1)^2+7(3t-1)^2=3n+13$ has an integer solution for any non-negative integer $n$.
\end{proof}

\section{The octagonal theorem of sixty}

Let $a_1,a_2,\dots,a_k$ be any positive integers. For an integer $n$, if the diophantine equation
$$
\Phi_{a_1,a_2,\dots,a_k}(x_1,x_2,\dots,x_k)=a_1P_8(x_1)+a_2P_8(x_2)+\cdots+a_kP_8(x_k)=n
$$
has an integer solution, then we say the sum $\Phi_{a_1,a_2,\dots,a_k}$ of generalized octagonal numbers represents $n$ and we write  $n\ra\Phi_{a_1,a_2,\dots,a_k}$. The sum $\Phi_{a_1,a_2,\dots,a_k}$ of generalized octagonal numbers is called universal over $\mathbb Z$ if it represents all non-negative integers.   In this section, we prove:
 
\begin{thm} The sum $\Phi_{a_1,a_2,\dots,a_k}$ of generalized octagonal numbers  is universal over $\mathbb Z$ if and only if it represents the integers 
$$
1, \ 2,\ 3,\  4,\  6,\  7,\  9, \ 12,\ 13,\  14, \ 18, \text{and} \ 60. 
$$ 
\end{thm}
To prove this, we need the following six lemmas: 

\begin{lem}\label{11214}
The quaternary sum  $\Phi_{1,1,2,14}$ of generalized octagonal numbers represents all positive integers except $60$.
\end{lem}

\begin{proof}
It suffices to show that the diophantine equation $x^2+y^2+2z^2+14t^2=3n+18$ has an integer solution $(x,y,z,t)=(a,b,c,d)\in\mathbb{Z}^4$ such that $abcd\not\equiv0 \ (\text{mod 3})$ for any non-negative integer $n \ne 60$. 
Note that an integer $n$ is represented by $\langle1,1,2\rangle$ if and only if it is not of the form $2^{2s+1}(8t+7)$ for any non-negative integers $s,t$. 
If $n\leq292$ and $n\not\eq60$, then one may directly check that the equation 
$$
x(3x-2)+y(3y-2)+2z(3z-2)+14t(3t-2)=n
$$
has an integer solution. Hence we may assume that $n\geq293$. One may easily show that there is an integer 
$$
\alpha \in \{ 3n+18-14\cdot 1^2, \  3n+18-14\cdot2^2, \  3n+18-14\cdot4^2\}
$$
 such that $\alpha \ra \langle1,1,2\rangle$. 
Furthermore, one may also show that there is an integer  
$$
\beta \in \{3n+18-14\cdot5^2, \ 3n+18-14\cdot7^2, \ 3n+18-14\cdot8^2\}
$$
such that $\beta \ra \langle1,1,2\rangle$. 
For all possible cases when both $\alpha$ and $\beta$ are squares of integers, one may easily check that the equation $x^2+y^2+2z^2=\alpha$ has an integer solution $(x,y,z)=(a,b,c)$ such that $abc\not\equiv0 \ (\text{mod 3})$. 
Therefore, we may assume that the equation $x^2+y^2+2z^2+14t^2=3n+18$ has an integer solution $(x,y,z,t)=(a,b,c,d)\in\mathbb{Z}^4$ such that $d\not\equiv0 \ (\text{mod 3})$ and $3n+18-14\cdot d^2$ is not a square of an integer. If $abc \not \equiv 0\ (\text{mod 3})$, then we are done. If $abc  \equiv 0\ (\text{mod 3})$, then without loss of generality,
we may assume that $b\equiv c\equiv0 \ (\text{mod 3})$ and $a \not \equiv 0 \ (\text{mod 3})$. Since $b^2+2c^2\not\eq0$, there are integers $e,f$ such that $e^2+2f^2=b^2+2c^2$ and $ef\not\equiv0 \ (\text{mod 3})$ by Theorem 9 of \cite{Jones}.  The lemma follows from this.
\end{proof}

\begin{lem}\label{1134}
The quaternary sum  $\Phi_{1,1,3,4}$  of generalized octagonal numbers represents all positive integers except $18$.
\end{lem}

\begin{proof}
It suffices to show that the diophantine equation $x^2+y^2+3z^2+4t^2=3n+9$ has an integer solution $(x,y,z,t)=(a,b,c,d)\in\mathbb{Z}^4$ such that $abcd\not\equiv0 \ (\text{mod 3})$ for any non-negative integer $n \ne 18$. If $n\leq46$ and $n \ne 18$, then one may directly check that such an integer solution exists. From now on, we assume that $n\ge 47$.  
 One may easily show that there is an integer $c \in \{2,5,7\}$ such that $3n+9-3\cdot c^2$ is represented by $\langle1,1,4\rangle$. 
Therefore, the equation $x^2+y^2+3z^2+4t^2=3n+9$ has an integer solution $(x,y,z,t)=(a,b,c,d)\in\mathbb{Z}^4$ such that $a^2+b^2+4d^2\ne0$ and $c\not\equiv0\ (\text{mod 3})$. If  $abd \not\equiv 0 \ (\text{mod 3})$, then we are done. Suppose that $a\equiv b\equiv d\equiv 0\ (\text{mod 3})$. 
If $\tau=\frac13\begin{pmatrix}1&2&4\\-2&-1&4\\-1&1&-1\end{pmatrix}$, then one may easily check that $\tau(a,b,d)^t=(a_1,b_1,d_1)^t$ is also an integer solution of 
\begin{equation} \label{yy} 
x^2+y^2+4t^2=3n+9-3c^2.
\end{equation} 
Therefore, there is a positive integer $m$ such that $\tau^m(a,b,d)^t=(a_m,b_m,d_m)^t$ is an integer solution of (\ref{yy}) such that $a_mb_md_m \not \equiv 0\ (\text{mod 3})$ or for any positive integer $m$, $\tau^{m}(a,b,d)^t=(a_{m},b_{m},d_{m})^t$ is an intger solution of (\ref{yy}) each of whose component is divisible by $3$.
Since there are only finitely many integer solution of (\ref{yy}) and $\tau$ has an infinite order, the latter is impossible unless $(a,b,d) $ is an eigenvector of $\tau$. 
Note that $(a,b,d)=(2t,12t,-5t)$ is an eigenvector of $\tau$ and 
$$
(2t)^2+(12t)^2+4(-5t)^2=(14t)^2+(6t)^2+4(2t)^2=248t^2,
$$ 
where $(14t,6t,2t)$ is not an eigenvector of $\tau$.
Therefore the equation $x^2+y^2+3z^2+4t^2=3n+9$ has always an integer solution $(x,y,z,t)=(a,b,c,d)\in\mathbb{Z}^4$ such that $abcd\not\equiv0 \ (\text{mod 3})$.
 This completes the proof.
\end{proof}

\begin{lem}\label{1233}
The quaternary sum  $\Phi_{1,2,3,3}$  of generalized octagonal numbers represents all positive integers except $12$.
\end{lem}

\begin{proof}
 Since the proof is quite similar to that of Theorem \ref{1237}, it is left to the reader. 
\end{proof}

\begin{lem}\label{11377}
The quinary sum $\Phi_{1,1,3,7,\alpha}$ of generalized octagonal numbers is universal over $\mathbb Z$ for any integer $\alpha\in\{7,9,10,11,13,14\}$.
\end{lem}

\begin{proof}
Since the proofs are quite similar to each other, we only provide the proof of the sum $\Phi_{1,1,3,7,7}$.
 
It suffices to show that the diophantine equation $x^2+y^2+3z^2+7t^2+7s^2=3n+19$ has an integer solution $(x,y,z,t,s)=(a,b,c,d,e)$ such that $abcde\not\equiv 0 \ (\text{mod 3})$ for any non-negative integer $n$.
If $n\leq12$, then one may directly check that such an integer solution exists. Therefore we assume that $n\geq 13$. 
Note that there are integers $d,e\in\{1,2\}$ such that $3n+19-7d^2-7e^2\equiv 0,1 \ (\text{mod 4})$.
Since the class number of $\langle 1,1,3 \rangle$ is one and $3n+19-7d^2-7e^2\equiv 2 \ (\text{mod 3})$, $3n+19-7d^2-7e^2$ is represented by  $\langle 1,1,3 \rangle$. 
Therefore the diophantine equation $x^2+y^2+3z^2+7t^2+7s^2=3n+19$ has an integer solution $(x,y,z,t,s)=(a,b,c,d,e)$ such that $d,e\in\{1,2\}$ and $3n+19-7d^2-7e^2\equiv 0,1 \ (\text{mod 4})$.
If $c\not\equiv0 \ (\text{mod 3})$, then we are done. Assume that $c\equiv0 \ (\text{mod 3})$. Without loss of generalitiy, we may assume that $b\equiv c \ (\text{mod 2})$, for $3n+19-7d^2-7e^2\equiv 0,1 \ (\text{mod 4})$. Then we have
$$
a^2+b^2+3c^2=a^2+\left(\frac{b+3c}{2}\right)^2+3\cdot\left(\frac{b-c}{2}\right)^2 =3n+19-7d^2-7e^2, 
$$
where $\left(\frac{b+3c}{2}\right)\left(\frac{b-c}{2}\right)\not\equiv 0 \ (\text{mod 3})$.
Therefore, the equation $(3x-1)^2+(3y-1)^2+3(3z-1)^2+7(3t-1)^2+7(3s-1)^2=3n+19$ has an integer solution for any non-negative integer $n$.
\end{proof}

\begin{lem}\label{11378}
The quinary sum $\Phi_{1,1,3,7,8}$ of generalized octagonal numbers is universal over $\mathbb Z$.
\end{lem}

\begin{proof}
It suffices to show that the diophantine equation $x^2+y^2+3z^2+7t^2+8s^2=3n+20$ has an integer solution $(x,y,z,t,s)=(a,b,c,d,e)$ such that $abcde\not\equiv 0 \ (\text{mod 3})$ for any non-negative integer $n$.
If $n\leq117$, then one may directly check that such an integer solution exists. Therefore we assume that $n\geq 118$. 
First, we show that the the diophantine equation $x^2+y^2+3z^2+7t^2+8s^2=3n+20$ has an integer solution $(x,y,z,t,s)=(a,b,c,d,e)$ such that $cd\not\equiv0 \ (\text{mod 3})$ and $3n+20-3c^2-7d^2$ is not a square of an integer.
Note that the class number of $\langle1,1,8 \rangle$ is one. Then one may easily show that for any integer $n\geq118$, 
$$
\begin{small}
\begin{aligned}
&3n+20-3\cdot1^2-7\cdot2^2,~3n+20-3\cdot2^2-7\cdot1^2\ra\langle1,1,8\rangle &\text{ if }& 3n+20\equiv 0 \ (\text{mod 4}),\\
&3n+20-3\cdot2^2-7\cdot2^2,~3n+20-3\cdot4^2-7\cdot2^2\ra\langle1,1,8\rangle &\text{ if }& 3n+20\equiv 1 \ (\text{mod 4}),\\
&3n+20-3\cdot1^2-7\cdot1^2,~3n+20-3\cdot5^2-7\cdot1^2\ra\langle1,1,8\rangle &\text{ if }& 3n+20\equiv 3 \ (\text{mod 4}),\\
&3n+20-3\cdot4^2-7\cdot4^2,~3n+20-3\cdot8^2-7\cdot4^2\ra\langle1,1,8\rangle &\text{ if }& 3n+20\equiv 2 \ (\text{mod 8}),\\
&3n+20-3\cdot1^2-7\cdot5^2,~3n+20-3\cdot5^2-7\cdot1^2\ra\langle1,1,8\rangle &\text{ if }& 3n+20\equiv 6 \ (\text{mod 16}),\\
&3n+20-3\cdot1^2-7\cdot1^2,~3n+20-3\cdot5^2-7\cdot5^2\ra\langle1,1,8\rangle &\text{ if }& 3n+20\equiv 14 \ (\text{mod 16}).
\end{aligned}
\end{small}
$$
Furthermore, in each case, at least one of two integers of the form $3n+20-3c^2-7d^2$ is not a square of an integer, for we are assuming  that $n\geq 118$. 
Therefore the equation $x^2+y^2+3z^2+7t^2+8s^2=3n+20$ has an integer solution $(x,y,z,t,s)=(a,b,c,d,e)$ such that $cd\not\equiv0 \ (\text{mod 3})$. If $abe\not\equiv0 \ (\text{mod 3})$, then we are done. Assume that $a\not\equiv0 \ (\text{mod 3})$ and $be\equiv0 \ (\text{mod 3})$. 
Since  $3n+20-3c^2-7d^2$ is not a square of an integers, $b^2+8s^2\neq0$. Therefore the  equation $(3x-1)^2+(3y-1)^2+3(3z-1)^2+7(3t-1)^2+8(3s-1)^2=3n+20$ has always an integer solution by Theorem 4.1 of \cite{oh}.
\end{proof}

\begin{lem}\label{113712}
The quinary sum $\Phi_{1,1,3,7,12}$ of generalized octagonal numbers is universal over $\mathbb Z$.
\end{lem}

\begin{proof}
It suffices to show that the equation $x^2+y^2+3z^2+7t^2+12s^2=3n+24$ has an integer solution $(x,y,z,t,s)$ such that $xyzts\not\equiv0 \ (\text{mod 3})$.
If $n\leq142$, then one may directly check that such an integer solution exists. Therefore we assume that $n\geq 143$. 
First, assume that $n=3m$. Note that the class number of $\langle1,4,6\rangle$ is one and 
every odd positive integer that is not divisible by $3$ is represented by $\langle1,4,6\rangle$.
Therefore $9m+22-7\cdot1^2$ or $9m+22-7\cdot8^2$ is represented by $\langle3,12,18\rangle$. 
Hence the equation
$$
(3x-1)^2+(3x+1)^2+3z^2+7t^2+12s^2=3(3m)+24
$$ 
has an integer solution such that $zts\not\equiv0 \ (\text{mod 3})$.
Assume that $n=3m+1$. Since $9m+25-7\cdot2^2$ or $9m+25-7\cdot7^2$ is represented by $\langle3,12,18\rangle$, the equation 
$$
(3x-1)^2+(3x+1)^2+3z^2+7t^2+12s^2=3(3m+1)+24
$$ 
has an integer solution such that $zts\not\equiv0 \ (\text{mod 3})$. Assume that $n=3m+2$.
Since $9m+28-7\cdot4^2$ or $9m+28-7\cdot5^2$ is represented by $\langle3,12,18\rangle$, the equation
$$
(3x-1)^2+(3x+1)^2+3z^2+7t^2+12s^2=3(3m+2)+24 
$$ 
has an integer solution such that $zts\not\equiv0 \ (\text{mod 3})$. 
Therefore the diophantine equation $(3x-1)^2+(3y-1)^2+3(3z-1)^2+7(3t-1)^2+12(3s-1)^2=3n+24$ has always an integer solution.
\end{proof}

\noindent {\it Proof of Theorem 3.1.} Without loss of generality, we may assume that $0<a_1\leq a_2 \leq \cdots \leq a_k$. 
Since $1\ra\Phi_{a_1,a_2,\dots,a_k}$, we have  $a_1=1$. 
Since $2\nra\Phi_{1}$, we have $a_2=1$ or $2$. 
Assume that $(a_1,a_2)=(1,1)$. Since $3 \nra \Phi_{1,1}$ and $3\ra  \Phi_{a_1,a_2,\dots,a_k}$, we have $1\leq a_3\leq3$.  
Assume that  $(a_1,a_2)=(1,2)$.   Since $4 \nra \Phi_{1,2}$ and  $4 \ra\Phi_{a_1,a_2,\dots,a_k}$,  we have $2\leq a_3\leq4$. 
Now, note that $a \nra \Phi_{a_1,a_2,a_3}$, where 
$$
a=\begin{cases}  4 \quad &\text{if $(a_1,a_2,a_3)=(1,1,1)$}, \\ 
 14 \quad &\text{if $(a_1,a_2,a_3)=(1,1,2)$}, \\ 
  7 \quad &\text{if $(a_1,a_2,a_3)=(1,1,3)$}, \\ 
   6 \quad &\text{if $(a_1,a_2,a_3)=(1,2,2)$}, \\ 
    9 \quad &\text{if $(a_1,a_2,a_3)=(1,2,3)$}, \\ 
     13 \quad &\text{if $(a_1,a_2,a_3)=(1,2,4)$}. \\  \end{cases}
$$
Therefore, $a_3 \le a_4 \le a$ for each possible case, where $a$ is the integer given above.

 If $(a_1,a_2,a_3,a_4)$ is one of possible quadruples given above except the quadruples
$(1,1,2,14)$, $(1,1,3,4)$, $(1,1,3,7)$ and $(1,2,3,3)$, then the quaternary sum $\Phi_{a_1,a_2,a_3,a_4}$ of generalized octagonal numbers  is universal over $\mathbb Z$ by \cite{sun} and Section 2.  
Now, note that $b \nra \Phi_{a_1,a_2,a_3,a_4}$ and $b \ra \Phi_{a_1,a_2,a_3,a_4,c}$ for any integer $c$ satisfying $a_4 \le c \le b$, where 
$$
b=\begin{cases}  60 \quad &\text{if $(a_1,a_2,a_3,a_4)=(1,1,2,14)$}, \\ 
 18 \quad &\text{if $(a_1,a_2,a_3,a_4)=(1,1,3,4)$}, \\ 
  12 \quad &\text{if $(a_1,a_2,a_3,a_4)=(1,2,3,3)$}. \\ 
    \end{cases}
$$
Furthermore, for each case, $\Phi_{a_1,a_2,a_3,a_4}$ represents all positive integers except $b$ by Lemmas \ref{11214}, \ref{1134} and \ref{1233}.  Therefore, the sum  $\Phi_{a_1,a_2,a_3,a_4,c}$ of generalized octagonal numbers is universal over $\mathbb Z$ for any $c$ such that $a_4 \le c \le b$. 
Note that $14\not\ra\Phi_{1,1,3,7}$. Therefore $7\leq a_5 \leq 14$ if $(a_1,a_2,a_3,a_4)=(1,1,3,7)$. The sum $\Phi_{1,1,3,7,d}$ of generalized octagonal numbers is universal over $\mathbb Z$ for any $d$ such that $7\leq d\leq14$ by Lemmas \ref{11377}, \ref{11378} and \ref{113712}. 
This completes the proof.  \qed

\end{document}